\documentclass[11pt,a4paper,english]{article}

\usepackage{amsmath}
\usepackage{latexsym}
\usepackage[english]{babel}
\usepackage{amsfonts}
\usepackage[T1]{fontenc}
\usepackage[latin1]{inputenc}
\usepackage{amsthm}
\usepackage{amssymb}
\usepackage{version}
\usepackage{fancyhdr}
\usepackage{caption}
\usepackage{natbib}
\usepackage{stmaryrd}
\usepackage{url}
\usepackage{hyperref}

\def \F {{\mathcal F}}

\def \H {{\mathcal H}}

\def \N {{\mathbb N}}
\def \P {{\mathbb P}}

\def \R {{\mathbb R}}

\def \I {{\mathbf 1}}
\def \bF {{\mathbb F}}

\def \bH {{\mathbb H}}

\newtheorem{theorem}{Theorem}[section]
\newtheorem{lemma}[theorem]{Lemma}
\newtheorem{definition}[theorem]{Definition}
\newtheorem{remark}[theorem]{Remark}
\newtheorem{proposition}[theorem]{Proposition}
\newtheorem{example}[theorem]{Example}

\newcommand{\ud}{\mathrm d}

\newcommand{\esp}[2][\mathbb E] {#1\left[#2\right]}

\newcommand{\condespf}[2][\F_t]       {\mathbb E\left.\left[#2\right|#1\right]}
\newcommand{\condespfo}[2][\F_0]       {\mathbb E\left.\left[#2\right|#1\right]}

\numberwithin{equation}{section}
\newcommand{\condespho}[2][\H_0]       {\mathbb E\left.\left[#2\right|#1\right]}
\newcommand{\condesph}[2][\H_t]       {\mathbb E\left.\left[#2\right|#1\right]}
\newcommand{\condesphh}[2][\H_{\tau-}]       {\mathbb E\left.\left[#2\right|#1\right]}
\newcommand{\condesphs}[2][\H_s]       {\mathbb E\left.\left[#2\right|#1\right]}

\newcommand{\mail}[1]{\href{mailto:#1}{\texttt{#1}}}

\title{GKW representation theorem and linear BSDEs under restricted information.
An application to risk-minimization.
}
\author{Claudia  Ceci\footnote{Dipartimento di Economia,
Universit{\`a} degli Studi ``G. D'Annunzio'' di Chieti-Pescara, Viale Pindaro 42,
I-65127 Pescara, Italy.
E-mail: \mail{ceci@sci.unich.it}.}
\and Alessandra Cretarola\footnote{Dipartimento di Matematica e Informatica, Universit{\`a} degli Studi di Perugia, via Vanvitelli, 1, I-06123 Perugia, Italy. E-mail: \mail{alessandra.cretarola@dmi.unipg.it}.}
\and Francesco Russo\footnote{Ecole Nationale Sup{\'e}rieure
 des Techniques Avancées,
ENSTA-ParisTech
Unit{\'e} de Math{\'e}matiques appliqu{\'e}es, 32 Boulevard Victor,
F-75739 Paris Cedex 15, France 
and INRIA Rocquencourt \& Cermics Ecole des Ponts,
 Projet MathFi. 
E-mail: \mail{francesco.russo@ensta-paristech.fr}.}}
\date{}
\begin{document}
\maketitle
\begin{abstract}
In this paper we provide Galtchouk-Kunita-Watanabe representation results 
in the case where there are restrictions on the available information. 
This allows to prove existence and uniqueness for linear backward stochastic
 differential equations driven by a general c{\`a}dl{\`a}g martingale under
 partial information. Furthermore, we discuss an application to
 risk-minimization where we extend the results of F{\"o}llmer and 
Sondermann (1986) to the partial information framework and we show how 
our result fits in the approach of Schweizer (1994).
\end{abstract}
\vspace{5mm}

\noindent {\it  Mathematics Subject Classification (2000):}
 60H10, 60H30, 91B28.\\
\noindent {\it Key words and phrases:} Backward stochastic differential 
equations, partial information, Galtchouk-Kunita-Watanabe decomposition,
 predictable dual projection, risk-minimiza-tion.

\vspace{2mm}
\section{Introduction}

\noindent
This paper provides two main contributions. First, we prove Galtchouk-Kunita-Watanabe representation results in the case where there are restrictions on the available information and we show an application to risk-minimization. Second, as an important consequence, we prove existence and uniqueness for linear backward stochastic differential equations (in short BSDEs) driven by a general c{\`a}dl{\`a}g martingale under partial information.\\
For BSDEs driven by a general c{\`a}dl{\`a}g martingale beyond the Brownian setting, there exist very few results in literature (see~\cite{ekh97} and more recently~\cite{bdm02} and~\cite{cfs08}, as far as we are aware). Here we study for the first time such a general case in the situation where there are restrictions on the available information, that represents an interesting issue arising in many financial problems. Mathematically, this  means to consider an additional filtration $\bH$ smaller than the full information flow $\bF$.
A typical example arises when ${\cal H}_t = {\cal F}_{(t-\tau)^+}$
where $\tau \in (0,T)$ is a fixed delay and $(t-\tau)^+:=\max\{0, t-\tau\}$ and $T$ denotes a time horizon.  \\ 
We start our investigation by considering BSDEs of the form
\begin{equation} \label{eq:bsde11} 
Y_t=\xi - \int_t^T Z_s \ud M_s -(O_T - O_t), \quad 0 \leq t \leq T,
\end{equation}
driven by a square-integrable (c{\`a}dl{\`a}g) martingale $M=(M_t)_{0\leq t\leq T}$, where $T>0$ is a fixed time horizon, $\xi \in L^2(\Omega,\F_T,\P;\R)$\footnote{The space $L^2(\Omega,\F_T,\P;\R)$ denotes the set of all $\F_T$-measurable real-valued random variables $H$ such that $\esp{|H|^2} = \int_\Omega |H|^2\ud \P < \infty$.} denotes the terminal condition and $O=(O_t)_{0\leq t \leq T}$ is a square-integrable $\bF$-martingale with $O_0 = 0$, satisfying a suitable orthogonality condition that we will make more precise in the next section. \\
We look for a solution $(Y,Z)$ to equation \eqref{eq:bsde11} under partial information, where $Y=(Y_t)_{0\leq t\leq T}$ is a c{\`a}dl{\`a}g $\bF$-adapted process and $Z=(Z_t)_{0\leq t\leq T}$ is an $\bH$-predictable process such that $\esp{\int_0^T|Z_t|^2\ud \langle M\rangle_t} < \infty$.\\
To this aim, we prove a Galtchouk-Kunita-Watanabe decomposition in the case where there are restrictions on the available information. 
More precisely, we obtain that every random variable $\xi \in L^2(\Omega,\F_T,\P;\R)$ can be uniquely written as 
\begin{equation} \label{eq:GKW1}
\xi = U_0 + \int_0^T H_t^\H\ud M_t + O_T, \quad \P-{\rm a.s.},
\end{equation}
where $H^\H=(H^\H)_{0\leq t\leq T}$ is an $\bH$-predictable process such that $\esp{\int_0^T|H_t^\H|^2\ud \langle M\rangle_t} < \infty$. To the authors' knowledge such a decomposition has not been proved yet in the existing literature. We will see that decomposition \eqref{eq:GKW1} allows to construct a solution to the BSDE \eqref{eq:bsde11} and ensures its uniqueness in this setting.\\
Moreover, we are able to provide an explicit characterization of the integrand process $H^\H$ given in decomposition  \eqref{eq:GKW1} in terms of the one appearing in the classical Galtchouk-Kunita-Watanabe decomposition, by using $\bH$-predictable (dual) projections. \\
Finally, we discuss a financial application. More precisely, we study the problem of hedging a contingent claim in the case where investors acting in the market have partial information. Since the market is incomplete we choose the risk-minimization approach, a quadratic hedging method which keeps the replicability constraint and relaxes the self financing condition, see~\cite{fs86} and~\cite{s01} for further details. As in~\cite{fs86} and~\cite{s94}, we consider the case where the price process is a martingale under the real world probability measure. In~\cite{fs86}, under the case of full information, the authors provide the risk-minimizing hedging strategy in terms of the classical Galtchouk-Kunita-Watanabe decomposition. Here, by using the Galtchouk-Kunita-Watanabe decomposition under partial information, we extend this result to the case where there are restrictions on the available information. Finally, thanks to the explicit representation of the integrand process $H^\H$ appearing in decomposition \eqref{eq:GKW1}, we find the same expression for the optimal strategy in terms of the Radon-Nikodym derivative of two $\bH$-predictable dual projections, that is proved in~\cite{s94}. \\
The paper is organized as follows. In Section \ref{setting} we give the definition of solution to BSDEs under partial information. Section \ref{sec:exun} is devoted to prove existence and uniqueness results for the solutions, which are obtained by applying the Galtchouk-Kunita-Watanabe decomposition adapted to the restricted information setting. The explicit representation of the integrand process $H^\H$ appearing in \eqref{eq:GKW1} can be found in Section \ref{sect:GKW}. Finally, an application to risk-minimization is given in Section \ref{sec:rm}.

 \section{Setting} \label{setting}

Let us fix a probability space $(\Omega,\F,\P)$ endowed with a filtration $\bF:= (\F_t)_{0\leq t\leq T}$, where $\F_t$
represents the full information at time $t$. We assume that $\F_T=\F$. Then we consider a subfiltration $\bH := (\H_t)_{0\leq t\leq T}$
of $\bF$, i.e. $\H_t \subseteq \F_t$, for each $t \in [0,T]$, corresponding to the available information level. We
remark that both filtrations are assumed to satisfy the usual hypotheses of completeness and right-continuity, see e.g.~\cite{pp}.\\
For simplicity we only consider the one-dimensional case. Extensions to several dimensions
are straightforward and left to the reader. The data of the problem are:
\begin{itemize}
\item an $\R$-valued square-integrable (c{\`a}dl{\`a}g) $\bF$-martingale $M=(M_t)_{0\leq t\leq T}$ with $\bF$-predictable
quadratic variation process denoted by $\langle M\rangle  = (\langle M,M\rangle)_{0\leq t\leq T}$;
\item a terminal condition $\xi \in L^2(\Omega,\F_T,\P;\R)$.
\end{itemize}

\begin{definition} \label{def:solBSDE}
A solution of the BSDE 
\begin{equation}\label{eq:bsde2}
Y_t=\xi - \int_t^T Z_s \ud M_s -(O_T - O_t), \quad 0 \leq t \leq T,
\end{equation}
with data $(\xi,\bH)$ under partial information, where $O=(O_t)_{0\leq t \leq T}$ is a square-integrable $\bF$-martingale with $O_0 = 0$, satisfying the orthogonality condition
\begin{equation} \label{eq:orthogcond}
\esp{O_T \int_0^T \varphi_t \ud M_t}=0,
\end{equation}
for all $\bH$-predictable processes $\varphi=(\varphi_t)_{0\leq t\leq T}$ such that  $\esp{\int_0^T|\varphi_t|^2\ud \langle M\rangle_t} < \infty$,
is a couple $(Y,Z)$ of processes with values in $\R \times \R$, satisfying (\ref{eq:bsde2}) such that
\begin{itemize}
\item $Y=(Y_t)_{0\leq t\leq T}$ is a c{\`a}dl{\`a}g $\bF$-adapted process;
\item $Z=(Z_t)_{0\leq t\leq T}$ is an $\bH$-predictable process such that $\esp{\int_0^T|Z_t|^2\ud \langle M\rangle_t} < \infty$.
\end{itemize}
\end{definition}

\begin{remark}
The orthogonality condition given in \eqref{eq:orthogcond} is weaker than the classical strong orthogonality condition, see e.g~\cite{js} or~\cite{pp}. Indeed, set $N_t=\int_0^t \varphi_s\ud M_s$, for each $t \in [0,T]$, where $\varphi$ is an  $\bH$-predictable process such that  $\esp{\int_0^T|\varphi_t|^2\ud \langle M\rangle_t} < \infty$. If 
$$
\langle O,M\rangle_t=0\quad \P-{\rm a.s.}, \quad \forall t \in [0,T],
$$
then
$$
\langle O,N\rangle_t=\int_0^t \varphi_s\ud \langle O,M\rangle_s=0\quad \P-{\rm a.s.}, \quad \forall t \in [0,T].
$$  
Consequently, $O \cdot N$ is an $\bF$-martingale null at zero, that implies 
$$
\esp{O_t N_t}=0, \quad \forall t \in [0,T],
$$ 
and in particular condition \eqref{eq:orthogcond}.
\end{remark}

\begin{remark}
Since for any $\bH$-predictable process $\varphi$, the process
$\I_{(0,t]}(s) \varphi_s$, with $t \leq T$, 
is $\bH$-predictable,  condition \eqref{eq:orthogcond} 
implies that for every $t \in [0,T]$
$$
\esp{O_T \int_0^t \varphi_s \ud M_s}=0,
$$
and 
by conditioning with respect to $\F_t$ (note that $O$ is an $\bF$-martingale), we have
$$
\esp{O_t \int_0^t \varphi_s \ud M_s}= \esp{\int_0^t \varphi_s \ud \langle M, O \rangle _s }= 0 \quad \forall t \in [0,T].
$$
From this last equality, we can argue that in the case of full information, i.e., $\H_t=\F_t$, for each $t \in [0,T]$, condition \eqref{eq:orthogcond} is equivalent to the strong orthogonality condition between $O$ and $M$ (see e.g. Lemma 2 and Theorem 36, Chapter IV, page 180 of~\cite{pp} for a rigorous proof). 
\end{remark}

\section{Existence and uniqueness for linear BSDEs under partial information} \label{sec:exun}

\noindent Our aim is to investigate existence and uniqueness of a solution to the BSDE \eqref{eq:bsde2} with data $(\xi, \bH)$ driven by the general martingale $M$ in the sense of Definition \ref{def:solBSDE}. This requires to 
prove a Galtchouk-Kunita-Watanabe representation result under restricted information.\\
\noindent 
We introduce
the linear subspace $\mathcal L_T^\H$ of $L^2(\Omega,\F_T,\P;\R)$ given by all
random variables $\eta$ of the form
\begin{equation} \label{def:subspace}
\left \{U_0 + \int_0^T \varphi_t \ud M_t\bigg{|}\ U_0 \in \H_0,\ \varphi\ {\rm is}\ \bH-{\rm predictable\ with}\ \esp{\int_0^T|\varphi_t|^2\ud \langle M\rangle_t} < \infty\right\}.
\end{equation}

\begin{lemma}\label{lem:closed}
The set $\mathcal L_T^\H$ is a closed subspace of $L^2(\Omega,\F_T,\P;\R)$.
\end{lemma} 
\begin{proof}
Let $U_0^n \in \H_0$ and $(\varphi^n)_{n \in \N}$, with $\varphi^n=(\varphi_t^n)_{0\leq t\leq T}$, be a sequence of
$\bH$-predictable processes satisfying $\esp{\int_0^T|\varphi_t^n|^2\ud \langle M\rangle_t} < \infty$ such that the sequence
\begin{equation*} \label{def:seq}
\eta^n = U_0^n + \int_0^T \varphi_t^n \ud M_t, \quad n \in \mathbb N,
\end{equation*}
converges to some random variable $\eta \in L^2(\Omega,\F_T,\P;\R)$, as $n$ goes to infinity. By taking the
conditional expectation with respect to $\H_0$, we have
$$
U_0^n = \condespho{\eta^n} \longrightarrow \condespho{\eta}, \quad {\rm as}\ n \to \infty.
$$
We set $U_0= \condespho{\eta}$. Since $(\eta^n-U_0^n)_{n \in \N}$ is a Cauchy sequence in $L^2(\Omega,\F_T,\P;\R)$, it follows
that
$$
\esp{\int_0^T(\varphi_{t}^n-\varphi_{t}^m)^2\ud \langle M\rangle_t} \longrightarrow 0, \quad {\rm as}\ n,m \to \infty.
$$
Consequently, $(\varphi^n)_{n \in \N}$ converges in $L^2(\Omega, \ud \langle M\rangle \otimes \ud \P)$\footnote{The space $L^2(\Omega, \ud \langle M\rangle \otimes \ud \P)$ denotes the set of all $\bF$-adapted processes $\varphi=(\varphi_t)_{0 \leq t \leq T}$ such that $$
\|\varphi\|_{L^2(\Omega, \ud \langle M\rangle \otimes \ud \P)}:=\left(\esp{\int_0^T|\varphi_t|^2\ud \langle M\rangle_t}\right)^{\frac{1}{2}} < \infty.
$$}
to some process $\varphi=(\varphi_t)_{0\leq t\leq T} \in L^2(\Omega, \ud \langle M\rangle \otimes \ud \P)$. 
Finally, since there is a subsequence converging $\ud \langle M\rangle \otimes \ud \P$-a.e., the limit $\varphi$ is
necessarily an $\bH$-predictable process.
\end{proof}

\begin{proposition} \label{prop:GKW}
Let $\xi \in L^2(\Omega,\F_T,\P;\R)$. There exists a unique decomposition of the form
\begin{equation} \label{eq:GKW}
\xi = U_0 + \int_0^T H_t^\H\ud M_t + O_T, \quad \P-{\rm a.s.},
\end{equation}
where $U_0 \in \H_0$, $H^\H$ is an $\bH$-predictable process such that $\esp{\int_0^T|H_t^\H|^2\ud \langle M\rangle_t} < \infty$ and $O$ is a square-integrable $\bF$-martingale with $O_0=0$ such that $\esp{O_T \cdot \eta}=0$, for every $\eta \in \mathcal L_T^\H$. Moreover $U_0= \condespho{\xi}$ and $\condespho{O_T}=0$.
\end{proposition}
\begin{proof}
The existence and uniqueness property of decomposition \eqref{eq:GKW} is clearly ensured
by the orthogonal projection of the random variable $\xi$ onto the space $\mathcal L_T^\H$, that is closed in virtue of Lemma \ref{lem:closed}. Since $(U_0 + \int_0^\cdot H_t^\H\ud M_t)$ is an $\bF$-martingale, by taking the conditional expectation of $\xi$
with respect to $\H_0$ in \eqref{eq:GKW}, we have
\begin{align*}
\condespho{\xi} & = \condespho{U_0 + \int_0^T H_t^\H\ud M_t} + \condespho{O_T}\\
& = \condespho{\condespfo{U_0 + \int_0^T H_t^\H\ud M_t}} + \condespho{\condespfo{O_T}}\\
&=\condespho{U_0},
\end{align*}
where in the last equality we have used the fact that $\condespfo{O_T}=O_0=0$.
Consequently $\condespho{O_T}=0$ and $U_0=\condespho{U_0}=\condespho{\xi}$. This concludes the proof.
\end{proof}

\begin{theorem}
Given data $(\xi,\bH)$, there exists a unique couple $(Y,Z)$ which solves the
BSDE \eqref{eq:bsde2} according to Definition \ref{def:solBSDE}.
\end{theorem}

\begin{proof} 
\noindent {\bf Existence.} Let $\mathcal L_T^\H$ be the linear subspace of $L^2(\Omega,\F_T,\P;\R)$ introduced in \eqref{def:subspace}. Given $\xi \in
L^2(\Omega,\F_T,\P;\R)$, we know by Proposition \ref{prop:GKW} that there exists a unique decomposition of
the form
$$
\xi = U_0 + \int_0^T H_t^\H\ud M_t + A_T, \quad \P-{\rm a.s.},
$$
where in particular $A$ is a square-integrable $\bF$-martingale with $A_0 = 0$ orthogonal to
all the elements in $\mathcal L_T^\H$. We use this result to construct a solution to the BSDE \eqref{eq:bsde2}. We
consider the orthogonal projection of $\xi \in
L^2(\Omega,\F_T,\P;\R)$ onto this space:
\begin{equation*} \label{eq:proj}
P_{\mathcal L_T^\H}(\xi):=U_0 + \int_0^T H_t^\H\ud M_t.
\end{equation*}
The couple $(U_0,H)$, where $U_0 \in \H_0$ and $H^\H$ is an $\bH$-predictable process in $L^2(\Omega,\ud \langle M\rangle \otimes \ud \P)$,
uniquely identifies the projection, that exists and it is well-defined since $\mathcal L_T^\H$ is closed. We set
$$
A_T:=\xi-P_{\mathcal L_T^\H}(\xi) \in (\mathcal L_T^\H)^\bot,
$$
where $(\mathcal L_T^\H)^\bot$ denotes the orthogonal subspace of
 $\mathcal L_T^\H$. Here $A_T$ corresponds to the
final value of a square-integrable $\bF$-martingale $A$ with zero initial value, that implies $\condespfo{\xi-U_0}=0$. Clearly, we have
$$
\mathcal L_T^\H \oplus (\mathcal L_T^\H)^\bot =L^2(\Omega,\F_T,\P).
$$
Now we define the process $Y$ as follows:
\begin{align*}
Y_t & :=\condespf{\xi}\\
&=\condespf{U_0 + \int_0^T H_t^\H\ud M_t + A_T}\\
&=\condespfo{U_0}+\int_0^t H_s^\H\ud M_s + A_t\\
&=Y_0+\int_0^t H_s^\H\ud M_s + A_t, \quad 0 \leq t\leq T,
\end{align*}
and we set $Z_t:=H_t^\H$ and $O_t:=A_t$, for every $t \in [0,T]$. Then we get
$$
Y_t = \xi -\int_t^T Z_s\ud M_s - (O_T-O_t), \quad 0 \leq t\leq T.
$$
\noindent {\bf Uniqueness.} Let $(Y,Z)$, $(Y',Z')$ be two solutions to the BSDE \eqref{eq:bsde2} under
partial information associated to the terminal condition $\xi \in L^2(\Omega,\F_T,\P;\R)$. We set
$(\bar Y,\bar Z) = (Y-Y',Z-Z')$. Then $(\bar Y,\bar Z)$ satisfies the BSDE
\begin{equation} \label{eq:bsde3}
\bar Y_t=-\int_t^T \bar Z_s \ud M_s -(\bar O_T - \bar O_t), \quad 0 \leq t \leq T,
\end{equation}
with final condition $\bar Y_T = 0$. In addition, we have set $\bar O := O-O'$ in \eqref{eq:bsde3}, where
$O$ and $O'$ denote the square-integrable $\bF$-martingales with $O_0 = O_0^{'} = 0$ satisfying the
orthogonality condition
$$
\esp{O_T\int_0^T \varphi_t \ud M_t}=\esp{O_T^{'}\int_0^T \varphi_t \ud M_t}=0,
$$
for all $\bH$-predictable processes $\varphi$ such that $\esp{\int_0^T|\varphi_t|^2\ud \langle M\rangle_t} < \infty$. Since $(\bar Y,\bar Z)$ is a solution of \eqref{eq:bsde3}, then
\begin{equation} \label{eq:barY}
\bar Y_t=\bar Y_0 + \int_0^t \bar Z_s \ud M_s + \bar O_t, \quad 0 \leq t \leq T.
\end{equation}
Since the process $\bar Y$ is an $\bF$-martingale such that $\bar Y_T = 0$, we have
$$
\bar Y_t=\condespf{\bar Y_T}=0, \quad  {\rm for\ all}\ t \in [0,T].
$$
Thus $Y_t=Y_t^{'}$ $\P$-a.s., for every $t \in [0,T]$.
Then we can rewrite \eqref{eq:barY} as follows
$$
0=\int_0^t \bar Z_s \ud M_s + \bar O_t, \quad 0 \leq t \leq T.
$$
By computing the predictable covariation of
$\int_0^{\cdot} \bar Z_s \ud M_s + \bar O$ and $\bar O$ and by taking the
expectation of both sides in the equality, for each $t \in [0,T]$, we obtain
\begin{align*}
0&=\int_0^t \bar Z_s \ud \langle M,\bar O\rangle_s + \langle\bar O\rangle_t\\
&=\esp{\int_0^t \bar Z_s \ud \langle M,\bar O\rangle_s}+\esp{\langle\bar O\rangle_t}.
\end{align*}
Since $\bar Z$ and $\bar O$ are differences of solutions to the BSDE \eqref{eq:bsde2}, then $\esp{\int_0^t \bar Z_s \ud \langle M,\bar O\rangle_s}=0$ for $t \in [0,T]$, and it follows
\begin{equation} \label{eq:espobar}
\esp{\langle\bar O\rangle_t}=0, \quad 0 \leq t \leq T.
\end{equation}
By Theorem 4.2 of~\cite{js}, since $\bar O$ is a square-integrable $\bF$-martingale null at zero, we have that $\bar{O}^2-\langle \bar O\rangle$ is an $\bF$-martingale null at zero. Then by \eqref{eq:espobar}
$$
\esp{\bar{O}_t^2}=\esp{\langle\bar O\rangle_t}=0,\quad 0 \leq t \leq T,
$$
that implies $\bar O_t^2=0$ $\P$-a.s. for every $t \in [0,T]$ and then $O_t=O_t^{'}$ $\P$-a.s. for every $t \in [0,T]$.
Now, let $Y$ be the unique solution of \eqref{eq:bsde2} for a certain $\bH$-predictable $Z$ such that $\esp{\int_0^T|Z_t|^2\ud \langle M\rangle_t} < \infty$, i.e.
\begin{equation} \label{eq:condy0}
Y_t=Y_0 + \int_0^t Z_s \ud M_s + O_t, \quad 0 \leq t \leq T.
\end{equation}
It only remains to prove that $Z$ is unique. 
For $t=T$ equation \eqref{eq:condy0} becomes
\begin{align*}
Y_T= \xi & = Y_0 + \int_0^T Z_s \ud M_s + O_T.
\end{align*}
By Proposition \ref{prop:GKW}, $Z_t = H_t^\H$ $\P$-a.s., for each $t \in [0,T]$ and then $Z$ is univocally determined. This concludes the proof.
\end{proof}

\section{Galtchouk-Kunita-Watanabe representation under partial information} \label{sect:GKW}

We now wish to provide an explicit characterization of the integrand process $H^\H$ appearing in the representation \eqref{eq:GKW} in terms of the one given in the classical Galtchouk-Kunita-Watanabe decomposition, by means in particular of the concept of $\bH$-predictable dual projection.\\
\noindent Let $\xi \in
L^2(\Omega,\F_T,\P;\R)$. We consider 
the well-known Galtchouk-Kunita-Watanabe decomposition of $\xi$ with respect to
$M$:
\begin{equation} \label{eq:GKWf}
\xi = \tilde U_0 + \int_0^T H_t^\F\ud M_t + \tilde O_T, \quad \P-{\rm a.s.},
\end{equation}
where $\tilde U_0 \in \F_0$, the integrand $H^\F=(H_t^\F)_{0\leq t \leq T}$ is an $\bF$-predictable process such that $\esp{\int_0^T|H_t^\F|^2\ud \langle M\rangle_t} < \infty$ and $\tilde O=(\tilde O_t)_{0\leq t \leq T}$ is a square-integrable $\bF$-martingale with $\tilde O_0=0$ such that $\langle \tilde O,M\rangle_t=0$, for every $t \in [0,T]$. 
Moreover, let us observe that $\tilde U_0=\condespfo{\xi}$.\\

\noindent 
In the sequel we will denote by ${}^pX$ the $\bH$-predictable projection of a (generic) integrable process $X=(X_t)_{0 \leq t\leq T}$, defined as the unique $\bH$-predictable process such that
$$
\condesphh{X_\tau \I_{\{\tau < \infty\}}}={}^pX_\tau \I_{\{\tau < \infty\}}\quad \P-{\rm a.s.}
$$
for every $\bH$-predictable stopping time $\tau$.\\
First we give a preliminary result under the additional assumption that the predictable quadratic variation
$\langle M\rangle$  of the $\bF$-martingale $M$ is an $\bH$-predictable process. In Theorem \ref{th:generalcase} we 
extend such result to the general case.
\begin{proposition} \label{prop:Hrepr}
Let $(\tilde U_0,H^\F,\tilde O_T)$ be the triplet corresponding to decomposition \eqref{eq:GKWf} of $\xi \in L^2(\Omega,\F_T,\P;\R)$. Suppose that the predictable quadratic variation
$\langle M\rangle$  of the $\bF$-martingale $M$ is an $\bH$-predictable process. Then
\begin{equation*} \label{eq:GKWh}
\xi = U_0 + \int_0^T H_t^\H\ud M_t + O_T, \quad \P-{\rm a.s.},
\end{equation*}
with
\begin{align}
U_0&=\condespho{\tilde U_0}, \label{eq:Uo}\\
H_t^\H&={}^p(H_t^\F), \quad 0\leq t \leq T \label{eq:hH},
\end{align}
and $O$ is a square-integrable $\bF$-martingale with $O_0=0$ such that $\esp{O_T \cdot \eta}=0$, for every $\eta \in \mathcal L_T^\H$. 
\end{proposition}  

\begin{proof}
Let
$$
\xi = \tilde U_0 + \int_0^T H_t^\F\ud M_t + \tilde O_T, \quad \P-{\rm a.s.}
$$
be the classical Galtchouk-Kunita-Watanabe decomposition of $\xi \in L^2(\Omega,\F_T,\P;\R)$.
By taking the expectation of $\xi$ with respect to $\H_0$, we have:
\begin{equation} \label{eq:expxi}
\condespho{\xi}=\condespho{\tilde U_0 + \int_0^T H_t^\F\ud M_t} + \condespho{\tilde O_T}.
\end{equation}
Since $(\tilde U_0 + \int_0^\cdot H_t^\F\ud M_t)$ is an $\bF$-martingale, it follows:
\begin{align*}
\condespho{\tilde U_0 + \int_0^T H_t^\F\ud M_t} & = \condespho{\condespfo{\tilde U_0 + \int_0^T H_t^\F\ud M_t}}\\
&=\condespho{\tilde U_0},
\end{align*}
so that we can rewrite \eqref{eq:expxi} as follows:
$$
\condespho{\xi}=\condespho{\tilde U_0}+\condespho{\tilde O_T}.
$$
Moreover, since $\tilde O$ is an $\bF$-martingale null at zero, we have
$$
\condespho{\tilde O_T}=\condespho{\condespfo{\tilde O_T}}=0.
$$
This implies equality \eqref{eq:Uo}. To prove equality \eqref{eq:hH}, we need to calculate the orthogonal
projection of $\xi$ onto the space $\mathcal L_T^\H$, see \eqref{def:subspace}. For the sake of brevity, we suppose that $\tilde U_0 = 0$. Thanks to Proposition \ref{prop:GKW}, this
means we need to check the following condition:
$$
\esp{\xi \int_0^T \varphi_t \ud M_t}=\esp{\int_0^T {}^p(H_t^\F) \ud M_t\int_0^T \varphi_t \ud M_t},
$$
for every $\bH$-predictable process $\varphi$ such that $\esp{\int_0^T|\varphi_t|^2\ud \langle M\rangle_t} < \infty$. Taking decomposition \eqref{eq:GKWf} into account, this corresponds
to the following equality:
\begin{equation}\label{eq:orthog}
\esp{\int_0^T H_t^\F \varphi_t \ud \langle M\rangle_t}=\esp{\int_0^T {}^p(H_t^\F) \varphi_t \ud \langle M\rangle_t},
\end{equation}
for every $\bH$-predictable process $\varphi$ such that $\esp{\int_0^T|\varphi_t|^2\ud \langle M\rangle_t} < \infty$. 
If we write the
process $\varphi$ as follows 
$$
\varphi = \varphi^+ - \varphi^-,
$$
where $\varphi^+$ and $\varphi^-$ denote the positive and the negative part of $\varphi$ respectively, and define the $\bF$-martingales
$$
R_t^+=\int_0^t \sqrt{\varphi^+_s} \ud M_s,\quad R_t^-=\int_0^t \sqrt{\varphi^-_s} \ud M_s, \quad 0 \leq t\leq T,
$$
equality \eqref{eq:orthog} is equivalent to the following relationships:
\begin{align*}
\esp{\int_0^T H_t^\F \ud \langle R^+\rangle_t}&=\esp{\int_0^T {}^p(H_t^\F)\ud \langle R^+\rangle_t}\\
\esp{\int_0^T H_t^\F \ud \langle R^-\rangle_t}&=\esp{\int_0^T {}^p(H_t^\F)\ud \langle R^-\rangle_t}.
\end{align*}
Hence, we can reduce the problem by assuming directly $\varphi_t = 1$ in \eqref{eq:orthog}, for each $t \in [0,T]$. Then, it is enough to prove the
equality
\begin{equation}\label{eq:orthog1}
\esp{\int_0^T H_t^\F\ud \langle M\rangle_t}=\esp{\int_0^T {}^p(H_t^\F)\ud \langle M\rangle_t}.
\end{equation}
Since $\langle M\rangle$ is $\bH$-predictable, Theorem VI.57 in~\cite{dm2}
guarantees that equality \eqref{eq:orthog1} holds, once we have the positivity of the process $H^\F$. By
writing
$$
H^\F=(H^\F)^+ - (H^\F)^-,
$$
and applying the above theorem to the positive and negative parts of $H^\F$, $(H^\F)^+$ and
$(H^\F)^-$ respectively, and to the associated $\bH$-predictable projections, we can get the
result by setting
$$
H^\H:={}^p(H^\F)={}^p((H^\F)^+) - {}^p((H^\F)^-).
$$
\end{proof}
\begin{example}
Let us consider the particular case where $M$ is a square-integrable $\bF$-martingale that is in addition a L\'evy process, $\F_t=\F_t^M$ and $\H_t=\F_{(t-\tau)^+}^M$, with $\tau \in (0,T)$ a fixed delay. We assume $\xi =h(M_T)\in L^2(\Omega,\F_T,\P;\R)$, for some measurable function $h:\R \longrightarrow \R$.\\
In this framework, 
by Lemma \ref{lem:repre} (see Appendix), we know that the integrand appearing in the Galtchouk-Kunita-Watanabe decomposition \eqref{eq:GKWf} can be written as
$$
H_t^\F = F(t,M_{t^-}), \quad t \in [0,T],
$$
where the function $F$ is such that the condition $\esp{\int_0^T|F(t,M_{t^-})|^2\ud  \langle M\rangle_t} < \infty$ is satisfied. Since in this case $\langle M\rangle$ is a deterministic process, we can apply Proposition \ref{prop:Hrepr} and get
$$
H_t^\H={}^pF(t,M_{t^-})=\mathbb E\left[F(t,M_{t^-})|\H_{t-}\right], \quad t \in [0,T].
$$ 
Then, it is easy to derive the following:
$$
H_t^\H=\left\{
\begin{array}{ll}
c(t,M_{(t -\tau)^-}) & \rm{if\ t > \tau}\\
c(t,M_0) & \rm{if\ t\leq \tau},
\end{array}
\right.
$$
where the function $c$ is given by
$$
c(t,y)=\int_\R F(t,y+z)\ud \rho_{t \wedge \tau}(z),
$$
with $\rho_t$ denoting the law of $M_t$, for every $t \in [0,T]$.
\end{example}

\subsection{The $\bH$-predictable dual projection} \label{sec:hpred}

\noindent It is possible to extend the result of Proposition \ref{prop:Hrepr} by using the concept of $\bH$-predictable dual projection. For reasons of clarity, we provide a self-contained discussion about this kind of projection in presence of more than one filtration.
Let $G=(G_t)_{0\leq t\leq T}$ be a c{\`a}dl{\`a}g $\bF$-adapted process of 
integrable variation, that is, $\esp{\|G\|_T}<\infty$. Here the process $\|G\|=(\|G\|)_{0 \leq t\leq T}$ defined, for each $t \in [0,T]$, by
$$
\|G\|_t(\omega)=\sup_{\Delta} \sum_{i=0}^{n(\Delta)-1}|G_{t_i+1}(\omega)-G_{t_i}(\omega)|, 
$$
where $\Delta = \{t_0=0<t_1<\ldots<t_{n(\Delta)}=t\}$ is a partition of $[0,t]$, 
denotes the total variation of the function $t \rightsquigarrow G_t(\omega)$. 
\begin{proposition} \label{th:hpred}
Let $G=(G_t)_{0\leq t\leq T}$ be a c{\`a}dl{\`a}g $\bF$-adapted process
of integrable variation.
Then there exists a unique $\bH$-predictable process $G^\bH=(G_t^\bH)_{0\leq t\leq T}$ of integrable variation, such that
\begin{equation*}\label{eq:hpredvar}
\esp{\int_0^T\varphi_t\ud G_t^\bH}=\esp{\int_0^T\varphi_t\ud G_t},
\end{equation*}
for every $\bH$-predictable (bounded) process $\varphi$. The process $G^\bH$ is called the $\bH$-{\rm predictable
dual projection} of $G$.
\end{proposition}

\begin{proof}
Without loss of generality, we can restrict our attention to the case where $G$ is an increasing process and prove the statement on the generators $\varphi$ of the form $\varphi_u=\I_{(s,t]}(u)\I_B$, with $B \in \H_s$ and $s,t \in [0,T]$ with $s < t$. Indeed, decomposing the process $G$ as $G = G^+ - G^-$, where both the
positive and negative parts of $G$ are assumed to be increasing integrable processes, we can suppose
$G$ to be increasing such that
$$
\esp{G_T}=\esp{\|G\|_T} < \infty.
$$
If $G$ is a c{\`a}dl{\`a}g, increasing, integrable $\bF$-adapted process, we will prove that there exists a unique increasing, integrable $\bH$-predictable process $G^\bH$ such that for every $s,t \in [0,T]$ with $s < t$ and $B \in \H_s$, the following relationship holds
$$
\esp{\I_B(G_t - G_s)}=\esp{\I_B(G_t^\bH - G_s^\bH)}.
$$
Let $\tilde G=(\tilde G_t)_{0\leq t\leq T}$ be the $\bH$-optional projection of $G$, such that for fixed times $t \in (0,T]$
$$
\tilde G_t=\condesph{G_t}\ \P-{\rm a.s.}.
$$
We observe that for every $s,t \in [0,T]$ with $s < t$ and $B \in \H_s$, we have
$$
\esp{\I_B(G_t - G_s)}=\esp{\I_B(\tilde G_t - \tilde G_s)}.
$$
Indeed,
\begin{align*}
\esp{\I_B(G_t - G_s)} & = \esp{\condesphs{\I_B(G_t - G_s)}} = \esp{\I_B\left(\condesphs{G_t} - \tilde G_s\right)}\\
& = \esp{\I_B\left(\condesphs{\tilde G_t} - \tilde G_s\right)} = \esp{\condesphs{\I_B(\tilde G_t - \tilde G_s)}}\\
& = \esp{\I_B(\tilde G_t - \tilde G_s)}.
\end{align*}
Furthermore, since $G$ is increasing, then $\tilde G$ is an $\bH$-submartingale, that is
$$
\condesphs{\tilde G_t} = \condesphs{\condesph{G_t}} = \condesphs{G_t} \ge \condesphs{G_s}=\tilde G_s, \quad 0\leq s \leq t \leq T.
$$
Thanks to Doob-Meyer Theorem on decomposition of submartingales, see e.g. Theorem 3.15 of~\cite{js}, there exists a unique increasing, integrable $\bH$-predictable process $G^\bH$ such that $\tilde G - G^\bH$ is an $\bH$-martingale, that is, for every $s,t \in [0,T]$ with $s < t$ and $B \in \H_s$, we have
$$
\esp{\I_B(\tilde G_t - \tilde G_s)}=\esp{\I_B(G_t^\bH - G_s^\bH)}.
$$
\end{proof}

\begin{remark} \label{rem:proj}
If $G$ is an 
$\bH$-predictable process of integrable variation and $X$ is 
an $\bF$-adapted process satisfying $\esp{\int_0^T X_t \ud G_t} < \infty$, then
$$
(X_t \ud G_t)^\bH={}^pX_t \ud G_t, \quad \P-{\rm a.s.},\ \mbox{for\ every}\ t \in [0,T].
$$ 
Indeed, by Theorem VI.57 in~\cite{dm2}, for any $\bH$-predictable (bounded) process $\varphi$ we can prove that
$$
\esp{\int_0^T \varphi_t X_t \ud G_t}=\esp{\int_0^T \varphi_t {}^pX_t \ud G_t}.
$$
\end{remark}

\subsection{Explicit representation results}

We now can apply the results of Subsection \ref{sec:hpred} to extend Proposition \ref{prop:Hrepr}.
Let $\mathcal P^\H$ and $\mathcal P$ be the $\bH$-predictable and $\bF$-predictable $\sigma$-field respectively. We consider the measures $\mu^\H$ (respectively $\mu$) defined on $\mathcal P^\H$ (respectively
$\mathcal P$) such that
\begin{equation}\label{eq:measureMuh}
\mu^\H((s,t] \times B)=\esp{\I_B(A_t^\bH - A_s^\bH)}, \quad B \in \H_s,\ \ s,t \in [0,T],\ s<t,
\end{equation}
where $A^\bH$ is the $\bH$-predictable dual projection of $A
:=(\int_0^t H_s^\F \ud \langle M\rangle_s)_{0\leq t \leq T}$,
that exists thanks to Theorem \ref{th:hpred}, and
\begin{equation}\label{eq:measureMuf}
\mu((u,v] \times F)=\esp{\I_F(\langle M\rangle_v^\bH - \langle M\rangle_u^\bH)}, \quad F \in \F_u,\ u,v \in [0,T],\ u<v.
\end{equation}
Here $H^\F$ is the integrand appearing in the Galtchouk-Kunita-Watanabe decomposition \eqref{eq:GKWf}.
\begin{lemma} \label{lem:abscont}
Let $\mu^\H$ and $\mu$ measures satisfying conditions \eqref{eq:measureMuh} and \eqref{eq:measureMuf} respectively.
Then $\mu^\H \ll \mu$ on $\mathcal P^\H$, that is, $\mu^\H$ is absolutely continuous with respect to the
restriction of $\mu$ on $\mathcal P^\H$.
\end{lemma}

\begin{proof}
By using the definition of absolute continuity, we wish to show that if whenever $\mu(E) = 0$ for $E \in \mathcal P^\H$, then $\mu^\H(E) = 0$. Let $\psi=(\psi_t)_{0\leq t\leq T}$ be a nonnegative $\bH$-predictable process such that
$$
\esp{\int_0^T \psi_t \ud  \langle M\rangle_t^\bH}=0.
$$
Then
$$
\esp{\int_0^T \psi_t \ud  \langle M\rangle_t}=0,
$$
that implies that $\psi = 0\ \ud \langle M \rangle 
\otimes \ud \P$ a.e., since $\psi$ is nonnegative. Finally
\begin{equation*}\label{eq:expAh}
\esp{\int_0^T \psi_t \ud  A_t^\bH}=\esp{\int_0^T \psi_t \ud  A_t}=\esp{\int_0^T \psi_t H_t^\F \ud  \langle M\rangle_t}=0.
\end{equation*}
\end{proof}

\noindent Since $\mu^\H \ll \mu$ on $\mathcal P^\H$, thanks to Lemma \ref{lem:abscont}, by the Radon-Nikodym theorem there
exists a $\mathcal P^\H$-measurable function $g$ on $[0,T] \times \Omega$ 
such that
$$
\mu^\H(E)=\int_E g(t,\omega) \ud \mu(t,\omega), \quad \forall E \in \mathcal P^\H.
$$
This allows to identify the process $H^\H$ as the Radon-Nikodym derivative:
\begin{equation} \label{eq:radnik}
H_t^\H(\omega):= \frac{\ud \mu^\H(t,\omega)}{\ud \mu(t,\omega)}\bigg{|}_{\mathcal P^\H}, \quad (t,\omega) \in [0,T] \times \Omega.
\end{equation}
Finally, we are ready to state the following theorem.
\begin{theorem} \label{th:general}
For any
nonnegative $\bF$-measurable process $H^\F$, the following equality holds
\begin{equation} \label{eq:hfcond}
\esp{\int_0^T \varphi_t H_t^\F \ud \langle M\rangle_t}=\esp{\int_0^T \varphi_t H_t^\H \ud \langle M\rangle_t},
\end{equation}
for every $\bH$-predictable process $\varphi$ such that $\esp{\int_0^T|\varphi_t|^2\ud \langle M\rangle_t} < \infty$. Here $H^\H$ is given by \eqref{eq:radnik}.
\end{theorem}
\begin{proof}
By relationship \eqref{eq:radnik} and definition of the measure $\mu$, see \eqref{eq:measureMuf}, we have for every $s,t \in [0,T]$ with $s < t$ and $B \in \H_s$
$$
\mu^\H((s,t] \times B) = \int_s^t\int_B H_u^\H(\omega)\ud \mu(u,\omega) =\esp{ \I_B \int_s^t  H_u^\H\ud \langle M\rangle_u^\bH}= \esp{ \I_B\int_s^t  H_u^\H\ud \langle M\rangle_u}.
$$ %
On the other hand, by \eqref{eq:measureMuh}
$$
\mu^\H((s,t] \times B) = \esp{\I_B(A_t^\bH - A_s^\bH)} = \esp{\I_B \int_s^t H_u^\F \ud \langle M\rangle_u}.
$$
If $\varphi$ is of the form $\varphi_u=\I_{(s,t]}(u)\I_B$, with $B \in \H_s$ and $s,t\in [0,T]$ with $s < t$, then the statement is proved since relationship \eqref{eq:hfcond} is verified on the generators of $\mathcal P^\H$.
\end{proof} 
\noindent 
We now give the analogous of Proposition \ref{prop:Hrepr}, without assuming that the process $\langle M \rangle$ is $\bH$-predictable.
\begin{theorem} \label{th:generalcase}
Let $(\tilde U_0,H^\F,\tilde O_T)$ be the triplet corresponding to decomposition \eqref{eq:GKWf} of $\xi \in L^2(\Omega,\F_T,\P;\R)$. Then
\begin{equation}\label{eq:GKWgeneral}
\xi = U_0 + \int_0^T H_t^\H\ud M_t + O_T, \quad \P-{\rm a.s.},
\end{equation}
with
\begin{align}
U_0&=\condespho{\tilde U_0}, \nonumber\\
H_t^\H&= \frac{\ud \mu^\H(t,\omega)}{\ud \mu(t,\omega)}\bigg{|}_{\mathcal P^\H}, \quad (t,\omega) \in [0,T] \times \Omega \nonumber,
\end{align}
where $\mu^\H$ and $\mu$ are given in \eqref{eq:measureMuh} and \eqref{eq:measureMuf} respectively, 
and $O$ is a square-integrable $\bF$-martingale with $O_0=0$ such that $\esp{O_T \cdot \eta}=0$, for every $\eta \in \mathcal L_T^\H$. 
\end{theorem}

\begin{proof}
We proceed as in the proof of Proposition \ref{prop:Hrepr} by observing that condition \eqref{eq:hfcond} plays the same role of condition \eqref{eq:orthog}.
\end{proof}
\noindent
In the next proposition we give a useful result which allows us to compute $H^\H$ as the Radon-Nikodym derivative of the $\bH$-predictable dual projection $A^\bH$ of the process $A=(\int_0^t H_s^\F \ud \langle M\rangle_s)_{0\leq t \leq T}$ with respect to the $\bH$-predictable dual projection $\langle M \rangle^\bH$ of the $\bF$-predictable quadratic variation $\langle M \rangle$.
 
 \begin{proposition}   \label{formula}
The process $A^\bH = (\int_0^.H_s^\F \ud \langle M\rangle_s)^\bH$ is absolutely continuous with respect to $\langle M \rangle^\bH$ and it is given by 
 $$
 A^\bH_t =  \int_0^t H^\H_s \ud \langle M \rangle_s^\bH, \quad 0 \leq t \leq T.
 $$
As a consequence 
 \begin{equation} \label{eq:radnik1}
H_t^\H = \frac{\ud A^\bH_t}{\ud \langle M \rangle_t^\bH}, \quad 0 \leq t \leq T.
\end{equation}
  \end{proposition}
 
 \begin{proof}
Set $\widetilde  A_t := \int_0^t H^\H_s \ud \langle M \rangle_s^\bH$, for each $t \in [0,T]$. It is sufficient to prove that 
$$\esp{\int_0^T \varphi_u \ud A_u}=\esp{\int_0^T \varphi_u  \ud \widetilde  A_u}$$
for every $\bH$-predictable (bounded) process $\varphi$. 
As before, we can consider $\varphi$ of the form $\varphi_u=\I_{(s,t]}(u)\I_B$, with $B \in \H_s$ and $s < t \in [0,T]$.

\noindent
Then by the definitions of the measure $\mu$ and $\mu^\H$, see \eqref{eq:measureMuf} and \eqref{eq:measureMuh}, and recalling  \eqref{eq:radnik} we get
$$
\esp{\int_0^T \varphi_u  \ud \widetilde  A_u} = \esp{ \I_B \int_s^t  H_u^\H\ud \langle M\rangle_u^\bH}=  \int_s^t\int_B H_u^\H(\omega)\ud \mu(u,\omega) = \mu^\H((s,t] \times B)=$$
$$\esp{\I_B(A_t^\bH - A_s^\bH)} = \esp{\I_B \int_s^t H_u^\F \ud \langle M\rangle_u} = \esp{\int_0^T \varphi_u \ud A_u}$$
which concludes the proof.
 \end{proof}
\begin{example}
Suppose that the process $\langle M \rangle$ is of the form
$$
\langle M \rangle_t=\int_0^ta_s\ud G_s,  \quad 0\leq t\leq T
$$
for some $\bF$-predictable process $a=(a_t)_{0\leq t\leq T}$ and an increasing deterministic function $G$. Then by Remark \ref{rem:proj} 
$$\langle M\rangle_t ^\bH =  \int_0^t {}^pa_s \ud G_s, \quad A^\bH_t =  \int_0^t {}^p(H_s^\F a_s)\ud G_s, \quad 0 \leq t \leq T,$$
and as a consequence of Proposition \ref{formula} we get
$$
H_t^\H=\frac{{}^p(H_t^\F a_t)}{{}^pa_t},\mathbb \quad 0\leq t\leq T.
$$

\end{example}

\begin{remark}
Let us observe that if the process $\langle M \rangle$ is $\bH$-predictable, then again by Remark \ref{rem:proj} 
$$\langle M\rangle_t ^\bH =  \langle M\rangle_t ,  \quad A^\bH_t =  \int_0^t {}^p(H_s^\F)\ud  \langle M\rangle_s, \quad 0 \leq t \leq T,$$
and by applying Proposition  \ref{formula} we obtain that
$$
H_t^\H={}^p(H_t^\F), \quad 0\leq t\leq T.
$$
 \end{remark}

\section{Risk-minimization under restricted information} \label{sec:rm}

In relation to the connection between risk-minimization under full and partial information respectively, we now show how our result obtained in Proposition \ref{formula} fits in the approach of~\cite{s94} of risk-minimization under restricted information.\\
On a probability space $(\Omega,\F,\P)$ we consider a financial market with  one riskless asset with (discounted) price $1$ and one risky asset whose (discounted) price  is given by a square-integrable (c{\`a}dl{\`a}g) martingale $M=(M_t)_{0\leq t\leq T}$ adapted to  a (large) filtration 
$\bF:= (\F_t)_{0\leq t\leq T}$. \\
We will study the problem of hedging a contingent claim,  whose final payoff is given by a random variable $\xi \in L^2(\Omega,\F_T,\P;\R)$, in the case where investors acting in the market can access only to the information flow   
$\bH := (\H_t)_{0\leq t\leq T}$ with $\H_t \subseteq \F_t$, for each $t \in [0,T]$.\\
We choose the risk-minimization approach to solve the above hedging problem. In the case of full information, in~\cite{fs86} the authors proved 
that there exists  a unique $\bF$-risk-minimizing hedging strategy $\phi^*=(\theta^*, \eta^*)$, where  $\theta^*=(\theta^*_t)_{0\leq t\leq T}$ is given by the integrand with respect to $M$ in the classical Galtchouk-Kunita-Watanabe  decomposition of $\xi$, i.e. $\theta^*= H^\F$ (see equation (\ref{eq:GKWf})).\\
In this section we extend this result to the case where there are restrictions on the available information, by using the Galtchouk-Kunita-Watanabe  decomposition
under partial information (see equation \eqref{eq:GKW}). More precisely, we prove that the $\bH$-risk-minimizing hedging strategy $\phi^{\H}=(\theta^{\H}, \eta^{\H})$ (see Definition \ref{def2} below) is such that $\theta^{\H}= H^\H$.

Risk-minimization under restricted information was studied in~\cite{s94} by using a different approach. We obtain the same explicit representation given in  Theorem 3.1 of~\cite{s94} by applying Proposition \ref{formula}.
About risk-minimization under partial information for defaultable markets via nonlinear filtering, we refer to~\cite{fs2012}. In particular, they consider the case where the contingent claim $\xi$ is $\H_T$-measurable, in which we can solve the risk-minimization problem by using the classical Galtchouk-Kunita-Watanabe  decomposition.\\
We now assume that the agent has at her/his disposal  the information flow $\bH$ about trading in stocks while a complete information about trading in the riskless asset. 

\begin{definition} \label{def1} 
 An $\bH$-strategy is a pair $\phi=(\theta, \eta)$ ($\theta_t$ is the number of shares  of the risky asset to be held at time $t$, while $\eta_t$  is the amount invested in the riskless asset at time $t$) where $\theta$ is $\bH$-predictable and $\eta$ is $\bF$-adapted and such that 
$$ 
\esp{\int_0^T  \theta^2_s \ud \langle M \rangle_s} < \infty
$$
and the value process $V(\phi) := \theta M + \eta$  satisfies 
$$
\esp{\left(\sup_{t \in [0,T]} | V_t(\phi) |\right)^2}< \infty.
$$
\end{definition}
\noindent For any $\bH$-strategy $\phi$, the associated cost process $C(\phi)$ is given by
$$
C_t(\phi):= V_t(\phi) - \int_0^t \theta_r \ud M_r, \quad 0 \leq t \leq T.
$$
Finally the $\bH$-risk process of $\phi$ is defined by
$$
R_t(\phi):= \condesph{(C_T(\phi) - C_t(\phi))^2}, \quad 0 \leq t \leq T.
$$
\begin{definition} \label{def2}
An $\bH$-strategy  $\phi=(\theta, \eta)$  is called $\bH$-risk-minimizing if $V_T(\phi) = \xi$ $\P$-a.s. and if for any $\bH$-strategy  $\psi$ such that $V_T(\psi) = \xi$ $\P$-a.s., we have
$R_t(\phi) \leq R_t(\psi)$ $\P$-a.s. for every $t \in [0,T]$.
\end{definition}

\begin{remark} \label{R1}
By Corollary 3.1 in~\cite{s94} we have  that if $\phi=(\theta, \eta)$  is an $\bH$-risk-minimizing strategy then $\phi$  is mean-self-financing, i.e. the cost process $C(\phi)$ is an $\bF$-martingale. 
Moreover,  if $\phi=(\theta, \eta)$  is 
a mean-self-financing $\bH$-strategy, then $V(\phi)$ is an $\bF$-martingale, hence for a given $\xi \in L^2(\Omega,\F_T,\P;\R)$, we have that $V_t(\phi) = \condespf{\xi}$, for every $t \in [0,T]$.
\end{remark}
\noindent To prove the main result of this section, we need the following Lemma.
\begin{lemma}\label{ORTO}
Let $O=(O_t)_{0\leq t \leq T}$ be a square-integrable $\bF$-martingale with $O_0 = 0$, satisfying the orthogonality condition
$$\esp{O_T \int_0^T \varphi_t \ud M_t}=0$$
for all $\bH$-predictable processes $\varphi=(\varphi_t)_{0\leq t\leq T}$ such that $\esp{\int_0^T|\varphi_t|^2\ud \langle M\rangle_t} < \infty$.
Then for any $t \in [0,T]$
\begin{equation*} \label{orto}
\condesph{(O_T - O_t) \int_t^T \varphi_s  \ud M_s}=0\quad \P-{\rm a.s.}.
\end{equation*}
\end{lemma}
 \begin{proof}
Since for any $\bH$-predictable process $\varphi$ 
$$
\I_{(t,T]}(s)\I_B \varphi_s, \quad B \in \H_t,\ t \in [0,T),
$$
is $\bH$-predictable, 
we get
$$
\esp{O_T \I_B  \int_t^T \varphi_s \ud M_s}= \esp{  \I_B  \condesph{ O_T  \int_t^T \varphi_s \ud M_s}}=0, \quad \forall B \in \H_t,
$$
and then
$$
\condesph{O_T\int_t^T \varphi_s \ud M_s}=0 \quad \P-{\rm a.s.}.
$$
\noindent Finally, let us observe that 
$$
\condesph{O_t \int_t^T \varphi_s  \ud M_s} = \condesph{O_t \int_0^T \varphi_s  \ud M_s} -  \condesph{O_t \int_0^t \varphi_s  \ud M_s}=0
$$
since 
$$ 
\condesph{O_t \int_0^T \varphi_s  \ud M_s} = \condesph{ \condespf{ O_t \int_0^T \varphi_s  \ud M_s}} = \condesph{O_t \int_0^t \varphi_s  \ud M_s},
$$
and this concludes the proof.
\end{proof} 
\noindent We are now in the position to provide an alternative proof to that given in~\cite{s94}, concerning the explicit representation for an $\bH$-risk-minimizing strategy, by applying the Galtchouk-Kunita-Watanabe decomposition under partial information and the representation result given in Proposition \ref{formula}. 
\begin{theorem} \label{ultimo} 
For every $\xi \in L^2(\Omega,\F_T,\P;\R)$, there exists a unique $\bH$-risk-minimizing strategy  $\phi^{\H}=(\theta^{\H}, \eta^{\H})$ 
 such that $\theta^{\H}= H^\H$, where $H^\H$ is given by 
\eqref{eq:radnik1} and $\eta_t^{\H} = \condespf{\xi} - \theta^{\H}_t M_t$, for every $t\in [0,T]$.
 \end{theorem}
 
 \begin{proof}
 The proof is similar to that of Theorem 2.4 of~\cite{s01} performed in the full information case.
 Let $\phi=(\theta,\eta)$ be a mean-self-financing $\bH$-strategy such that $V_T(\phi) = \xi$ $\P$-a.s.. Hence, by computing the Galtchouk-Kunita-Watanabe decomposition under partial information, see \eqref{eq:GKWgeneral}, we have
\begin{gather*}
C_T( \phi) - C_t( \phi) = V_T( \phi) - V_t( \phi) - \int_t^T  \theta_s \ud M_s = \xi - V_t( \phi) - \int_t^T  \theta_s \ud M_s\\
=U_0 + \int_0^T H_t^\H \ud M_t + O_T - \int_t^T  \theta_s \ud M_s  - V_t( \phi),
\end{gather*}
where $H^\H$ is given by \eqref{eq:radnik1}.
Since $V_t(\phi) = E [ \xi | \F_t]$, for every $t\in [0,T]$, see Remark \ref{R1}, we get that
$$
V_t( \phi) = U_0 + \int_0^t H_s^\H \ud M_s+ O_t
$$
and
$$ 
C_T(\phi) - C_t( \phi) = \int_t^T  \{ H_s^\H  - \ \theta_s\}  \ud M_s + O_T - O_t.
$$
By similar computation we get that 
$$
C_T(\phi^\H) - C_t( \phi^\H) =  O_T - O_t.
$$
Finally 
\begin{equation*}
\begin{split}(C_T(\phi) - C_t( \phi) )^2 & =  \Big (C_T(\phi^\H) - C_t( \phi^\H)  \Big)^2 + \Big ( \int_t^T  \{ H_s^\H  - \theta_s\}  \ud M_s  \Big )^2\\
& \qquad  + 2 (O_T - O_t) \int_t^T  \{ H_s^\H  - \ \theta_s\}  \ud M_s
\end{split}
\end{equation*}
and by Lemma \ref{ORTO} we obtain that 
$$
R_t(\phi) = R_t(\phi^\H) + \condesph{\left( \int_t^T  \{ H_s^\H  -  \theta_s\}  \ud M_s\right)^2} \geq R_t(\phi^\H).
$$
Hence $\phi^\H$ is $\bH$-risk-minimizing.
If some other $\phi$ is also $\bH$-risk-minimizing then 
$$
\condespho{\left( \int_0^T  \{ H_s^\H  -  \theta_s\}  \ud M_s \right)^2} = \condespho{\int_0^T  \{ H_s^\H  -  \theta_s\} ^2 \ud \langle M \rangle_s} = 0,
$$
which implies $H^\H=\theta$. 
Since $V_t(\phi) = V_t(\phi^\H) = \condespf{\xi}$ for each $t \in [0,T]$, we also obtain $\phi= \phi^\H.$
\end{proof}

\noindent In the rest of the section we investigate the case where there is a relationship between the information flow $\bH$ and the filtration generated by the stock price $M$, that we denote by $\bF^M$. 
A possible choice is the assumption that investors acting in the market have access only to the information contained in past asset prices, that is $\bH = \bF^M$. Such a situation has been studied for instance in~\cite{f2000} and~\cite{c06} for stock price dynamics with jumps.
In the sequel we will  assume $\bH \subseteq \bF^M$, which takes also into account, for instance, the case where the asset price is only observed at discrete times or with a fixed delay $\tau \in (0,T)$, i.e. $\H_t=\F^M_{(t-\tau)^+}$, for every $t \in (0,T)$.
\medskip

\noindent In a such particular case, when in addition $\xi \in L^2(\Omega,\F^M_T,\P;\R)$, we can find an $\bH$-risk minimizing strategy, $\phi^\H= (\theta^\H, \eta^\H)$,  where $\theta^H$ is $\bH$-predictable and $\eta^H$ is $\bF^M$-adapted, while in the general case $\eta^\H$ has been taken $\bF$-adapted.   This means that we study the situation where the agent has at her/his disposal the information flow $\bH\subseteq \bF^M$  about trading in stocks and the filtration $\bF^M$ about trading in the riskless asset, and when $\bH= \bF^M$ the same information flow.

More precisely, from now on we restrict ourself to consider $\bH$-strategies $\phi= (\theta, \eta)$ as in Definition \ref{def1} where $\eta$ is chosen $\bF^M$-adapted. 

\begin{remark} \label{R2}
Let us observe that given an $\bH$-strategy $\phi= (\theta, \eta)$, the associated value process $V(\phi) := \theta M + \eta$  turns out to be $\bF^M$-adapted. 
By Corollary 3.1 in~\cite{s94}, we have  that if $\phi=(\theta, \eta)$  is an $\bH$-risk-minimizing strategy according to this new definition, then $\phi$  is mean-self-financing, i.e. the cost process $C(\phi)$ is an $\bF^M$-martingale. 
Moreover,  if $\phi=(\theta, \eta)$  is 
a mean-self-financing $\bH$-strategy, then  $V(\phi)$ is an $\bF^M$-martingale, hence $V_t(\phi) = E [\xi | \F^M_t]$, for every $t \in [0,T]$.
\end{remark}
\medskip

\noindent We are now ready to give the following result.
 
 \begin{theorem} \label{ultimo1} 
For every $\xi \in L^2(\Omega,\F^M_T,\P;\R)$, there exists a unique $\bH$-risk-minimizing strategy  $\phi^{\H}=(\theta^{\H}, \eta^{\H})$ 
 such that $\theta^{\H}= H^\H$, where $H^\H$ is given 
by \eqref{eq:radnik1} and $\eta^{\H} = E [\xi | \F^M_t]  - \theta^{\H}_t M_t$, for every $t \in [0,T]$.
 \end{theorem}
 
 \begin{proof}
Since $\xi$ is $\F^M_T$-measurable, by decomposition \eqref{eq:GKWgeneral} we obtain that
$$
\xi = \mathbb E[\xi | \F^M_T ] = U_0 + \int_0^T H_t^\H \ud M_t + \mathbb E[ O_T | \F^M_T],
$$
where $H^\H$ is given by \eqref{eq:radnik1}.
Set $\hat O_t := \mathbb E[O_t | \F^M_t]$, for each $t \in [0,T]$. It is known that $\hat O$ is an $\bF^M$-martingale and 
$$
\mathbb E\left[ \hat O_T  \int_0^T H_t^\H \ud M_t \right] = \mathbb E\left[\mathbb E\left[ \hat O_T |  \F^M_T\right] \int_0^T H_t^\H \ud M_t\right] = \mathbb E\left[  O_T  \int_0^T H_t^\H \ud M_t\right]=0.
$$
Therefore we obtain the Galtchouk-Kunita-Watanabe decomposition of $\xi$ under restricted information with respect to the filtration  $\bF^M$, that is
$$
\xi = U_0 + \int_0^T H_t^\H \ud M_t + \hat O_T,\quad \P-{\rm a.s.}. 
$$
The rest of the proof 
follows from Theorem \ref{ultimo} by replacing the filtration $\bF$ by $\bF^M$ and the $\bF$-martingale $O$ by the $\bF^M$-martingale $\hat O$. 
\end{proof}


\begin{center}
{\bf Acknowledgements}
\end{center}

\noindent Part of the paper was written during the stay of the 
second named author  
 at INRIA Paris-Rocquencourt (Projet MathFi).
The second and third named authors would like to thank
Prof. Agn\`es Sulem for stimulating discussions. 
The work of the first and second named authors was partially supported by PRIN 2008 Project ``Probability and Finance''. The work of the third named author  
was  partially supported by the ANR
Project MASTERIE 2010 BLAN-0121-01 and it was partially 
carried out during his visit to the Bernoulli Center (EPFL Lausanne).

\appendix

\

\begin{center}
      {\bf APPENDIX}
    \end{center}

\section{Technical results}

\noindent Recall that $M=(M_t)_{0\leq t\leq T}$ is a square-integrable $\bF$-martingale and assume that $\bF = \bF^M:=(\F_t^M)_{0 \leq t \leq T}$, i.e. the information flow $\bF$ coincides with the canonical filtration $\bF^M$ of $M$.
\begin{lemma} \label{lem:repre}
Let $M$ be a L{\'e}vy process and $\xi=h(M_T) \in L^2(\Omega,\F_T^M,\P;\R)$ for some measurable function $h:\R \rightarrow \R$. Then, there exists a measurable function $F:[0,T] \times \R \rightarrow \R$ such that
$$
\xi = \esp{\xi} + \int_0^T F(s,M_{s-})\ud M_s + \tilde O_T, \quad \P-\mbox{a.s.},
$$
where $\tilde O=(\tilde O_t)_{0\leq t\leq T}$ is a square-integrable $\bF^M$-martingale null at zero such that $\langle \tilde O,M\rangle_t=0$, for every $t \in [0,T]$. Moreover, the following integrability condition is satisfied
$$
\esp{\int_0^T|F(s,M_{s-})|^2\ud \langle M\rangle_s}< \infty.
$$
\end{lemma}

\begin{proof}
If $\xi$ is given as a Fourier transform of $M_T$, that is, the function $h$ is of the form
\begin{equation} \label{eq:hfunction}
h(x) = \int_\R e^{iax} \ud \nu(a),\quad \mbox{for\ all\ } x \in \R,
\end{equation}
where $\nu$ is a finite measure, the result is contained in Proposition 4.3 of~\cite{gor}, which was an adaptation of~\cite{hubalek}. \\
As a consequence, the thesis follows once we show the existence of a sequence $(h_n)_{n \in \N}$ of functions of the kind \eqref{eq:hfunction} such that
\begin{equation} \label{eq:converg}
\esp{\left|h_n(M_T)-h(M_T)\right|^2} \xrightarrow[n \to \infty]{} 0.
\end{equation}

\noindent To see that, denoting  by $F_n(t,M_{t-})$ the integrand in the Galtchouk-Kunita-Watanabe decomposition
of $h_n(M_T)$, $n\in \N$,  we can proceed as in the proof of Lemma \ref{lem:closed} and we get that  the sequence $ (F_n(t,M_{t-}) )_{n \in \N}$  converges in $L^2(\Omega, \ud \langle M\rangle \otimes \ud \P)$ to  the integrand   $H^\F $ in the Galtchouk-Kunita-Watanabe decomposition 
of $h(M_T)$.  Now, there is a subsequence converging
  $\ud \langle M\rangle \otimes \ud \P$-a.e. 
to the $\bF$-predictable process
  $(H_t^\F)_{0\leq t \leq T}$  and for almost all $t \in [0,T]$,
$H_t^\F$ is $\sigma(M_{t-})$-measurable. Finally this implies 
the existence of a measurable function $F: [0,T] \times \R \rightarrow \R$ 
such that $H_t^\F = F(t,M_{t-})$.

\noindent It remains to show the existence of a sequence $(h_n)_{n\in \N}$ of functions
of type \eqref{eq:hfunction} 
verifying  (\ref{eq:converg}). 
If $\rho_T$ is the law of $M_T$, \eqref{eq:converg}, translates into
\begin{equation} \label{AR2}
\int_\R (h_n(y) - h(y))^2 \ud \rho_T(y) \xrightarrow[n \to \infty]{} 0.
\end{equation}
Since $\rho_T$ is a finite non-negative measure, it is well-known that
the space of smooth functions with compact support is dense in
$L^2(\rho_T)$. This implies that the Schwartz space ${\mathcal S}(\R)$
of the fast decreasing functions is dense in $L^2(\rho_T)$.
Let $(h_n)_{n \in \N}$ belong to ${\mathcal S}(\R)$ such that \eqref{AR2}, 
and consequently, \eqref{eq:converg} holds.
Since the inverse Fourier transform  ${\mathcal F}^{-1}$ maps
 ${\mathcal S}(\R)$ into itself,  then we observe that 
$h_n$ are of the type 
\eqref{eq:hfunction} with $\nu(\ud a) ={\mathcal F}^{-1} h_n(a) \ud a$.

\end{proof}

\bibliographystyle{plain} 
\bibliography{CCRbiblio}

\end{document}